\documentclass[11pt]{article}
\usepackage[a4paper,margin=28mm]{geometry}
\usepackage[T1]{fontenc}
\usepackage{lmodern,amsmath,amssymb,amsthm,mathtools}
\usepackage{microtype,enumitem,booktabs,longtable,array}
\usepackage[colorlinks=true,linkcolor=blue,citecolor=blue,urlcolor=blue]{hyperref}
\usepackage{fancyhdr}
\newtheorem{theorem}{Theorem}[section]
\newtheorem{proposition}[theorem]{Proposition}
\newtheorem{lemma}[theorem]{Lemma}
\newtheorem{corollary}[theorem]{Corollary}
\theoremstyle{definition}\newtheorem{example}[theorem]{Example}
\theoremstyle{remark}

\def\m{\frak m}
\def\gld{{\rm gl.dim}}
\newcommand{\AP}{\mathcal A}
\newcommand{\Flat}{\mathcal F}
\newcommand{\Proj}{\mathcal P}
\newcommand{\Spec}{\operatorname{Spec}}
\newcommand{\Max}{\operatorname{Max}}
\newcommand{\Hom}{\operatorname{Hom}}
\newcommand{\Ext}{\operatorname{Ext}}
\newcommand{\im}{\operatorname{im}}
\newcommand{\ann}{\operatorname{ann}}
\newcommand{\eps}{\varepsilon}
\newcommand{\N}{\mathbb N}\newcommand{\Z}{\mathbb Z}
\hypersetup{pdftitle={Two remarks on Almost Projective Modules},pdfauthor={}}
\title{Two remarks on Almost Projective Modules}

\author{Xiaolei Zhang\thanks{Corresponding author} \thanks{School of Mathematics and Statistics, Tianshui Normal University 
	Tianshui 741000, China;	E-mail: zxlrghj@163.com }\and Wei Qi\thanks{School of Mathematics and Statistics, Tianshui Normal University 
	Tianshui 741000, China; 
	E-mail: qwrghj@126.com} \and Dechuan Zhou \thanks{
	School of Mathematics and Physics, Southwest University of Science and Technology  Mianyang 621010, China;
	E-mail: dechuan11119@sina.com}}
\begin{document}

\maketitle
\begin{abstract}
For a commutative ring $R$, an almost projective module is a module whose localization at every maximal ideal is free, with no restriction on rank. We prove that every almost projective $R$-module is projective if and only if $R$ is semilocal. Next, we turn to investigate when every $R$-module has an almost projective envelope. We show that under certain conditions, every $R$-module has an almost projective envelope if and only if it is coherent and locally Artinian. We also  show that the conditions can not be omitted in terms of  dual numbers rings over von Neumann regular rings.
\end{abstract}
\noindent\textbf{Keywords.} Almost projective module; envelope; semilocal ring; coherent ring; locally Artinian ring.\\
\textbf{2020 Mathematics Subject Classification.} 13C10, 13D07, 16D80.

\section{Introduction and principal results}
Throughout this paper, $R$ is a commutative ring with identity and all $R$-modules are unitary.  We always denote by $\Max(R)$ the spectrum of all maximal ideals of $R$, $K.\dim(R)$ the Krull dimension of $R$ and w.\gld$(R)$ the weak global dimension of $R$.

The notion of locally free modules, in the sense that $R$-modules  $M$ satisfying that $M_\m$ is a free $R_\m$-module for any  maximal ideal $\m$  of $R,$ is a key terminology to the areas of both commutative algebra and algebra geometry. Recently, Zhou et al. \cite{Zhou2020} found that an $R$-module $M$ is locally free if and only if  \begin{equation}\label{eq:definition}
	\Ext^1_R(M,N)=0
	\quad\text{for every }\mathfrak m\in\Max R
	\text{ and every }R_{\mathfrak m}\text{-module }N,
\end{equation} and the later is is the definition of an almost projective module in \cite{Zhou2020}.  
Write $\AP(R)$, $\Flat(R)$, and $\Proj(R)$ for the classes of almost projective, flat, and projective modules. We use ``locally free'' only in the stalkwise sense $M_{\mathfrak m}$ free for every maximal ideal. For infinitely generated modules this does not assert freeness on an open neighborhood of each point.
The authors also defined the almost global dimension $A.\gld(R)$ of a ring $R$ in terms of almost projective modules,   and obtained that rings $R$ with $A.\gld(R)=0$ are exactly von Neumann regular rings, and domains with $A.\gld(R)\leq 1$ are exactly almost Dedekind domains.

Our first main result gives an  answer to when every almost projective module is projective?
\begin{theorem}[Semilocal characterization]\label{thm:main-semilocal}
For a commutative ring $R$, the following are equivalent:
\begin{enumerate}[label=(\roman*)]
\item Every almost projective $R$-module is projective.
\item Every countably presented almost projective $R$-module is projective.
\item Every countably presented $R$-module $M$ satisfying
\[
 M_{\mathfrak m}\cong R_{\mathfrak m}^{(\N)}
 \quad(\mathfrak m\in\Max R)
\]
is projective.
\item $R$ is semilocal.
\end{enumerate}
\end{theorem}
The proof is given in Section~\ref{sec:semilocal}. Its descent step uses Perry's theorem on arbitrary projectivity \cite[Theorem 9.6]{Perry}; the nonsemilocal construction and its nonprojectivity verification are given explicitly. In particular, coherence and finite-dimensional hypotheses play no role.

Recently, Zhou et al. \cite{Zhou2021} studied the existence of almost projective covers, and showed every $R$-module has an almost projective cover if and only if $R$ is a  locally perfect ring, i.e., rings $R$ satisfying $R_\m$ is a   perfect ring for any  maximal ideal $\m$  of $R.$

It is well known that every $R$-module has a projective envelope if and only if $R$ is an Artinian ring (see  \cite{Asensio}). If $R$ is a domain or has w.\gld$(R)<\infty$, then every $R$-module has a flat envelope if and only if w.\gld$(R)\leq 2$ (see \cite{Asensio,E81,M86}). However, the characterizations of rings with existences of flat envelopes for all $R$-modules remain open  in general case. Another of the main motivation is to investigate the existence of almost projective envelopes.

\begin{proposition}\label{prop:positive}
	Suppose $R$ is coherent locally Artinian. Each of the following conditions ensures that every module has an almost projective envelope:
	\begin{enumerate}[label=(\roman*)]
		\item $R$ is reduced;
		\item $R$ has finite weak global dimension;
		\item $R$ is Noetherian;
		\item $R$ is semilocal.
	\end{enumerate}
\end{proposition}

A natural question is that can we omit these conditions in the above result?  To resolve this question, we investigate the existence of almost projective envelopes over  $R=B[\eps]/(\eps^2)$, where $B$ is a von Neumann regular ring.

\begin{theorem}[Dual-number classification]\label{thm:main-dual}
Let $B$ be a commutative von Neumann regular ring and put $R=B[\eps]/(\eps^2)$. Then $R$ is coherent and locally Artinian. The following are equivalent:
\begin{enumerate}[label=(\roman*)]
\item Every $R$-module has an almost projective envelope.
\item Every module $B/\mathfrak p$, regarded as an $R$-module annihilated by $\eps$, has an almost projective envelope, for $\mathfrak p\in\Spec B$.
\item Every residue field $B/\mathfrak p$ is projective over $B$.
\item $B$ is a finite product of fields.
\end{enumerate}
\end{theorem}

Consequently, we give a example that  every module does not have an almost projective envelope over coherent locally Artinian rings.

\section{Semilocality characterizes projectivity of all almost projectives}\label{sec:semilocal}
In this section, we will characterize when every almost projective module is projective.
\subsection{Almost projectivity and localization}
\begin{lemma}\label{lem:ext}
Let $S\subseteq R$ be multiplicatively closed and let $N$ be an $S^{-1}R$-module. For every $R$-module $M$ and every $i\geq0$ there is a natural isomorphism
\[
 \Ext^i_R(M,N)\cong
 \Ext^i_{S^{-1}R}(S^{-1}M,N).
\]
\end{lemma}
\begin{proof}
Choose a projective resolution $P_\bullet\to M$. Localization is exact and carries projectives to projectives, so $S^{-1}P_\bullet\to S^{-1}M$ is a projective resolution. Localization--restriction adjunction gives an isomorphism of complexes
\[
 \Hom_R(P_\bullet,N)
 \cong\Hom_{S^{-1}R}(S^{-1}P_\bullet,N).
\]
Taking cohomology proves the assertion. No finite-presentation assumption on $M$ is needed: we localize the first argument and assume the second argument already is localized. This is distinct from commuting localization with an arbitrary Ext group.
\end{proof}

\begin{proposition}\label{prop:localfree}
An $R$-module $M$ is almost projective if and only if $M_{\mathfrak m}$ is projective, equivalently free, over $R_{\mathfrak m}$ for every maximal ideal $\mathfrak m$. Consequently
\[
 \Proj(R)\subseteq\AP(R)\subseteq\Flat(R).
\]
The class $\AP(R)$ is closed under direct sums, direct summands, and extensions.
\end{proposition}
\begin{proof}
Apply Lemma~\ref{lem:ext} with $S=R\setminus\mathfrak m$. Vanishing of $\Ext^1_{R_{\mathfrak m}}(M_{\mathfrak m},-)$ on all modules is precisely projectivity. Arbitrary projective modules over a local ring are free by Kaplansky's theorem \cite{Kaplansky}. Flatness can be tested at maximal ideals, giving the second inclusion. Localization of a projective module is projective, giving the first.

Direct sums and summands of projectives remain projective after localization. For an extension of two almost projective modules, localization gives an extension whose quotient is projective; it splits. The localized middle term is therefore projective. This proves all closure assertions.
\end{proof}

\begin{proposition}\label{prop:flat-ap}
If every $R_{\mathfrak m}$ is a perfect ring, then $\AP(R)=\Flat(R)$. In particular this holds if $R$ is locally Artinian.
\end{proposition}
\begin{proof}
A perfect ring is a ring over which every flat module is projective; the equivalence with the projective-cover definition is Bass's theorem \cite{Bass}. If $M$ is flat, each $M_{\mathfrak m}$ is flat over the perfect ring $R_{\mathfrak m}$ and hence projective. Proposition~\ref{prop:localfree} applies. Conversely every almost projective module is flat. Artinian rings are perfect.
\end{proof}

\subsection{The semilocal direction}
\begin{proposition}\label{prop:semilocal-suff}
If $R$ is semilocal, then $\AP(R)=\Proj(R)$.
\end{proposition}
\begin{proof}
List the maximal ideals as $\mathfrak m_1,\ldots,\mathfrak m_t$, and let
\[
 S=\prod_{i=1}^t R_{\mathfrak m_i}.
\]
This is flat over $R$, because the product is finite. It is faithfully flat: if $N\otimes_R S=0$, then every $N_{\mathfrak m_i}=0$, hence $N=0$. To justify the last assertion, a nonzero element of $N$ has a proper annihilator contained in a maximal ideal, and survives localization there.

For $M\in\AP(R)$, the $S$-module
\[
 S\otimes_R M\cong\prod_{i=1}^t M_{\mathfrak m_i}
\]
is projective. Indeed, modules over a finite product ring are tuples of modules over the factors, and each component here is free. Arbitrary projectivity descends along faithfully flat maps by \cite[Theorem 9.6]{Perry}. Therefore $M$ is projective.
\end{proof}
The use of arbitrary, rather than only finitely presented, projectivity in this descent theorem is essential. No uniform bound on the local ranks has been assumed.

\subsection{Comaximal ideals in a nonsemilocal ring}
We give the elementary ideal form of a standard comaximal-family lemma; a lattice version appears in \cite[Lemma 21]{DumitrescuEpure}.
\begin{lemma}\label{lem:comaximal}
If $\Max R$ is infinite, there are proper finitely generated ideals $I_1,I_2,\ldots$ with $I_i+I_j=R$ for $i\ne j$.
\end{lemma}
\begin{proof}
Call a proper finitely generated ideal isolated if it lies in exactly one maximal ideal. Suppose first that every proper finitely generated ideal is contained in an isolated one. Given isolated ideals with distinct maximal ideals $\mathfrak m_1,\ldots,\mathfrak m_n$, choose another maximal ideal $\mathfrak p$ and elements $c_i\in\mathfrak p\setminus\mathfrak m_i$. An isolated ideal containing $(c_1,\ldots,c_n)$ has a new maximal ideal. Induction produces the required family.

Otherwise choose a proper finitely generated $A_0$ contained in no isolated ideal. Every proper finitely generated overideal of $A_0$ then lies in at least two maximal ideals. Recursively choose proper finitely generated ideals $A_n,B_n\supseteq A_{n-1}$ with $A_n+B_n=R$: for distinct maximal overideals choose $a+b=1$ with $a$ in one and $b$ in the other, and adjoin $a$ and $b$, respectively. For $j<n$ we have $A_j\subseteq B_n$, hence $B_j+B_n=R$. Take $I_n=B_n$.
\end{proof}

\subsection{A stabilization consequence of projectivity}
The next lemma makes the nonprojectivity argument self-contained; it is the elementary part of the usual Mittag--Leffler obstruction.
\begin{lemma}\label{lem:stabilize}
Let
\[
 F_0\xrightarrow{u_0}F_1\xrightarrow{u_1}F_2\longrightarrow\cdots
\]
be a sequence of finite free modules, with colimit $M$. Write $u_{ij}:F_i\to F_j$ for its transition maps. If $M$ is projective, the descending sequence
\[
 \im\big(\Hom_R(F_n,R)\longrightarrow\Hom_R(F_0,R)\big)
 \qquad(n\geq0)
\]
is eventually constant.
\end{lemma}
\begin{proof}
Let $j_n:F_n\to M$ be the colimit maps. The canonical surjection
\[
 q:\bigoplus_{n\geq0}F_n\longrightarrow M
\]
has a section $s$ if $M$ is projective. Since $F_0$ is finitely generated, $sj_0$ has image in $\bigoplus_{n\leq N}F_n$ for some $N$. Let $\pi$ be the projection onto this finite subsum, and let $t:\bigoplus_{n\leq N}F_n\to F_N$ be the sum of the transition maps. Put $a=t\pi sj_0$. Then
\[
 j_Na=j_0=j_Nu_{0N}.
\]
There is $L\geq N$ with $u_{NL}a=u_{0L}$: equality in a directed colimit can be checked at a later stage for each of the finitely many basis elements of $F_0$.

For $k\geq L$, set $a_k=t\pi sj_k:F_k\to F_N$. We have $a_ku_{0k}=a$. For any $\lambda:F_L\to R$, the map
\[
 \mu=\lambda u_{NL}a_k:F_k\longrightarrow R
\]
satisfies $\mu u_{0k}=\lambda u_{0L}$. Thus the image at stage $L$ is contained in the image at stage $k$. The reverse inclusion follows from $u_{0k}=u_{Lk}u_{0L}$. The images stabilize at $L$.
\end{proof}

\subsection{An explicit nonprojective almost projective module}
\begin{theorem}\label{thm:construction}
Every nonsemilocal commutative ring admits a countably presented, nonprojective module $M$ with
\[
 M_{\mathfrak m}\cong R_{\mathfrak m}^{(\N)}
 \quad\text{for every }\mathfrak m\in\Max R.
\]
\end{theorem}
\begin{proof}
Choose $I_n=(a_{n1},\ldots,a_{nk_n})$ from Lemma~\ref{lem:comaximal}, padding by zero generators to ensure $k_n\geq2$. Let $v_n$ be the corresponding column in $R^{k_n}$. Define finite free modules and maps recursively by
\[
 F_0=R,\qquad F_n=F_{n-1}\otimes_R R^{k_n},\qquad
 u_{n-1}(x)=x\otimes v_n,
\]
and put $M=\varinjlim F_n$. The standard telescope sequence
\[
 0\longrightarrow\bigoplus_{n\geq0}F_n
 \xrightarrow{\,1-\mathrm{shift}\,}\bigoplus_{n\geq0}F_n
 \longrightarrow M\longrightarrow0
\]
shows that $M$ is countably presented. The first map is injective even when the individual transition maps are not: inspecting components from the first nonzero component proves this directly.

\emph{Local freeness.} Fix $\mathfrak m$. At most one $I_n$ is contained in $\mathfrak m$, by pairwise comaximality. Discard a finite initial segment so that $I_nR_{\mathfrak m}=R_{\mathfrak m}$ at every remaining step. In the local ring $R_{\mathfrak m}$ at least one entry of $v_n$ is a unit. Elementary basis changes identify $v_n$ with a first basis vector. Hence each remaining localized transition is a split injection of finite free modules with finite free cokernel. The localized colimit is the direct sum of its first module and these successive cokernels. The cokernel ranks are positive, since $k_n\geq2$. Their countable direct sum is free of rank $\aleph_0$.

\emph{Failure of projectivity.} Under $\Hom_R(F_0,R)=R$, the image of $\Hom_R(F_n,R)$ is exactly
\begin{equation}\label{eq:ideal-products}
 J_n=I_1I_2\cdots I_n.
\end{equation}
Indeed, the coordinates of $u_{0n}(1)$ are all products $a_{1j_1}\cdots a_{nj_n}$, and arbitrary linear functionals give their generated ideal. These ideals form a strictly descending chain. Choose a maximal ideal $\mathfrak m$ containing $I_{n+1}$. All earlier $I_i$ become the unit ideal at $\mathfrak m$, so
\[
 (J_n)_{\mathfrak m}=R_{\mathfrak m},\qquad
 (J_{n+1})_{\mathfrak m}=(I_{n+1})_{\mathfrak m}\ne R_{\mathfrak m}.
\]
Thus $J_{n+1}\ne J_n$. Lemma~\ref{lem:stabilize} rules out projectivity of $M$. Its local freeness gives almost projectivity by Proposition~\ref{prop:localfree}.
\end{proof}

\begin{proof}[Proof of Theorem~\ref{thm:main-semilocal}]
The implications (i)$\Rightarrow$(ii)$\Rightarrow$(iii) are immediate. Theorem~\ref{thm:construction} proves (iii)$\Rightarrow$(iv), and Proposition~\ref{prop:semilocal-suff} proves (iv)$\Rightarrow$(i).
\end{proof}

\subsection{Consequences and the proposed finite-product criterion}\label{sec:consequences}
\begin{example}[A connected counterexample to the proposed equivalence]\label{ex:semilocal-domain}
Let
\[
 T=\{n\in\Z:\gcd(n,6)=1\},\qquad R=T^{-1}\Z.
\]
The only maximal ideals are $2R$ and $3R$, so Theorem~\ref{thm:main-semilocal} gives $\AP(R)=\Proj(R)$. But $R$ is a domain with two maximal ideals. A domain cannot be a product of two nonzero rings, because such a product has nontrivial idempotents. A one-factor representation as a local ring is also impossible. Thus $R$ is not a finite product of local rings.
\end{example}
Finite products of local rings remain sufficient, since they are semilocal. The missing condition in the proposed converse is idempotent lifting, not projectivity of locally free modules.

\begin{corollary}\label{cor:semiperfect}
For a commutative ring $R$, the following are equivalent:
\begin{enumerate}[label=(\roman*)]
\item $R$ is a finite product of local rings;
\item every almost projective module is projective and idempotents lift modulo $J(R)$;
\item $R$ is semiperfect.
\end{enumerate}
\end{corollary}
\begin{proof}
Theorem~\ref{thm:main-semilocal} converts (ii) into semilocality plus idempotent lifting, the definition of semiperfectness. For a commutative semilocal ring, $R/J(R)$ is a finite product of fields. Lifting its orthogonal idempotent decomposition gives a finite decomposition of $R$ into local rings. Conversely a finite product of local rings is semiperfect.
\end{proof}

\begin{corollary}\label{cor:zero-dim}
For a zero-dimensional commutative ring $R$, every almost projective module is projective if and only if $R$ is a finite product of local rings. If $R$ is locally Artinian, these conditions are equivalent to $R$ being Artinian.
\end{corollary}
\begin{proof}
In a zero-dimensional ring, $R/\sqrt0$ is von Neumann regular. If $R$ is also semilocal, this quotient is a finite product of fields. Idempotents lift modulo the nil ideal $\sqrt0$, giving a finite product of local rings. Apply Theorem~\ref{thm:main-semilocal}.

For a semilocal locally Artinian ring, every ascending or descending chain of ideals stabilizes: localize at its finitely many maximal ideals, take a stabilization index at each Artinian localization, and choose the maximum. Equality of ideals is detected at all maximal localizations. Thus $R$ is Artinian. Conversely an Artinian commutative ring is semilocal and locally Artinian.
\end{proof}

\begin{example}[A rank-one obstruction over $\Z$]
Let $p_1,p_2,\ldots$ be distinct primes and set
\[
 L=\bigcup_{n\geq1}\frac{1}{p_1\cdots p_n}\Z\subseteq\mathbb Q.
\]
At a prime $p_i$, $L_{(p_i)}=p_i^{-1}\Z_{(p_i)}$, a free rank-one module. At every other prime it is $\Z_{(p)}$. Hence $L$ is almost projective. It is not finitely generated, since denominators of finitely many rational generators involve only finitely many primes. A projective rank-one module over a domain is finitely generated: in a dual basis, only finitely many coordinate functionals are nonzero on a fixed nonzero element; every remaining functional is zero on the whole rank-one torsion-free module. The dual-basis formula then gives finitely many generators. Thus $L$ is not projective. This illustrates the obstruction without the variable-rank construction of Section~\ref{sec:semilocal}.
\end{example}

\section{On the Existence of Almost Projective (Pre)envelopes}
In this section, we will investigate when every module has an almost projective (pre)envelope.

\subsection{Preenvelopes and the surviving coherence condition}\label{sec:preenvelopes}
For a class $\mathcal C$, a map $f:M\to C$, $C\in\mathcal C$, is a $\mathcal C$-preenvelope if every map from $M$ to a member of $\mathcal C$ factors through $f$. It is an envelope if, in addition, every endomorphism $h$ of $C$ with $hf=f$ is an automorphism.

\begin{proposition}\label{prop:coherence}
	If every $R$-module has an almost projective preenvelope, then $R$ is coherent. More generally every product of almost projective modules is almost projective under this hypothesis.
\end{proposition}
\begin{proof}
	Let $X=\prod_{i\in I}A_i$ with $A_i\in\AP(R)$, and take a preenvelope $u:X\to A$. Each coordinate map $X\to A_i$ factors through $u$. Combining the resulting maps gives $v:A\to X$ with $vu=1_X$. Thus $X$ is a summand of $A$ and is almost projective. Apply this with all $A_i=R$. All products of copies of $R$ are flat, and Chase's product criterion \cite{Chase} gives coherence.
\end{proof}

\begin{theorem}\label{thm:preenvelopes}
	For a locally Artinian commutative ring $R$, the following are equivalent:
	\begin{enumerate}[label=(\roman*)]
		\item $R$ is coherent;
		\item every $R$-module has an almost projective preenvelope;
		\item every $R$-module has a flat preenvelope.
	\end{enumerate}
\end{theorem}
\begin{proof}
	Proposition~\ref{prop:flat-ap} equates the two classes, and Proposition~\ref{prop:coherence} proves necessity. We recall the solution-set argument for sufficiency.
	
	Given $M$, put $\kappa=\max(|R|,|M|,\aleph_0)$. If $g:M\to F$ has flat target, its image lies in a pure submodule $F'\subseteq F$ of cardinality at most $\kappa$. To construct $F'$, start with the submodule generated by $g(M)$. At each stage adjoin one finite tuple of solutions for each finite system of linear equations with constants in the current submodule which is solvable in $F$. There are at most $\kappa$ such systems at each stage. The union after countably many stages has size at most $\kappa$ and is pure, by the finite-equation criterion for purity. A pure submodule of a flat module is flat.
	
	Choose a set of representatives of flat modules of cardinality at most $\kappa$, and take the product $P$ of one such module for every map from $M$ to a representative. Coherence makes $P$ flat by Chase's theorem. The evaluation map $M\to P$ factors every map to a flat module, by the preceding small-submodule construction. It is a flat preenvelope.
\end{proof}

\begin{corollary}\label{cor:reduction}
	Over a locally Artinian ring, almost projective envelopes and flat envelopes are exactly the same morphisms. 
\end{corollary}
The missing ingredient in passing from Theorem~\ref{thm:preenvelopes} to envelopes is minimality, not a solution set. It is therefore insufficient simply to construct an evaluation map into a product of local free modules.

\subsection{Further envelope characterizations}\label{sec:further}

We use the classical projective-envelope theorem: for a commutative ring, every module has a projective envelope if and only if the ring is Artinian \cite[Corollary 3.6]{Asensio}. This is an external input; it is not inferred by localizing an almost projective envelope.
\begin{theorem}\label{thm:semilocal-env}
	For a semilocal commutative ring $R$, every module has an almost projective envelope if and only if $R$ is Artinian.
\end{theorem}
\begin{proof}
	By Theorem~\ref{thm:main-semilocal}, the target classes are equal: $\AP(R)=\Proj(R)$. Apply the projective-envelope theorem.
\end{proof}

\begin{proposition}\label{prop:positive}
	Suppose $R$ is coherent locally Artinian. Each of the following conditions ensures that every module has an almost projective envelope:
	\begin{enumerate}[label=(\roman*)]
		\item $R$ is reduced;
		\item $R$ has finite weak global dimension;
		\item $R$ is Noetherian;
		\item $R$ is semilocal.
	\end{enumerate}
\end{proposition}
\begin{proof}
	Local Artinianness implies that every prime ideal is maximal. If $R$ is reduced, it is therefore von Neumann regular; all its modules are almost projective and their identity maps are envelopes.
	
	For (ii), a nonfield Artinian local ring $A$ has infinite weak global dimension. Here is a direct verification. Its residue field has a minimal free resolution with finitely generated terms. Such a resolution cannot terminate in positive degree: an injective last differential would have all entries in the maximal ideal, and multiplication of a last basis vector by a nonzero socle element would give a nonzero element of its kernel. Nor is the residue field projective unless $A$ is a field. Tensoring the infinite minimal resolution with the residue field gives nonzero Tor in arbitrarily high degrees. Thus every localization of $R$ must be a field when its weak global dimension is finite. A ring all of whose maximal localizations are fields is von Neumann regular, so (i) applies.
	
	For (iii), a zero-dimensional Noetherian ring is Artinian. For (iv), Corollary~\ref{cor:zero-dim} gives Artinianness. The projective-envelope theorem and $\AP=\Proj$ for Artinian rings complete these two cases.
\end{proof}

A quasi-Frobenius (QF) ring is an Artinian self-injective ring. We record the following classical characterization of monomorphic flat envelopes, due to Asensio Mayor--Mart\'inez Hern\'andez and developed further by Saor\'in \cite{Asensio,Saorin}: every module over a commutative ring has a monomorphic flat envelope precisely when, for every maximal ideal $\mathfrak m$, the ring $R_{\mathfrak m}$ is either a field or a QF ring which is projective as an $R$-module. We use this published theorem only in the next corollary.

\begin{theorem}\label{thm:qf}
	Let $R$ be coherent and locally QF. Then every $R$-module has an almost projective envelope if and only if every maximal ideal with nonfield localization is an isolated point of $\Spec R$.
\end{theorem}
\begin{proof}
	First every module embeds in a flat module. For each maximal ideal choose an injective envelope
	\[
	M_{\mathfrak m}\hookrightarrow E_{\mathfrak m}
	\quad\text{over }R_{\mathfrak m}.
	\]
	Over a QF ring injectives are projective, so $E_{\mathfrak m}$ is flat as an $R$-module. Coherence makes $\prod_{\mathfrak m}E_{\mathfrak m}$ flat. The composite of the localization diagonal with these inclusions is injective: a nonzero element survives at some maximal localization. Hence any flat preenvelope of $M$ is injective, by factoring this embedding through it.
	
	Consequently, existence of flat envelopes for all modules is equivalent here to existence of monomorphic flat envelopes. Almost projective and flat envelopes coincide by Proposition~\ref{prop:flat-ap}. The quoted classification reduces the question to projectivity of $R_{\mathfrak m}$ at each nonfield point.
	
	For completeness, over a zero-dimensional ring the localization map $R\to R_{\mathfrak m}$ is surjective. If $s\notin\mathfrak m$, the clopen set $D(s)$ corresponds to an idempotent $e$, and $s$ is a unit in $eR$; its inverse there represents $s^{-1}$ in $R_{\mathfrak m}$. Here the clopen description follows from absolute flatness of $R/\sqrt0$ and idempotent lifting modulo $\sqrt0$. Thus $R_{\mathfrak m}$ is a cyclic quotient. If it is projective, it is a direct factor $eR$, whose support is the singleton $\{\mathfrak m\}$, so that point is open. Conversely an isolated point gives an idempotent direct factor equal to $R_{\mathfrak m}$, hence a projective module. This proves the required equivalence.
\end{proof}

The square-zero construction has a more general counterpart. It gives a complete criterion for envelopes of simple modules over coherent locally Artinian rings, without self-injectivity assumptions.

\begin{lemma}\label{lem:supported-element}
	Let $R$ be coherent and zero-dimensional, let $\mathfrak m\in\Spec R$, and put $A=R_{\mathfrak m}$ and $I=\ker(R\to A)$. The localization map $\pi:R\to A$ is surjective, and $I$ is generated by idempotents. If $\mathfrak m$ is not isolated, then
	\[
	\ann_R(I)=0.
	\]
\end{lemma}
\begin{proof}
	For completeness, the spectrum of a zero-dimensional ring is a Stone space: it is the spectrum of its von Neumann regular reduced quotient. Its clopen subsets correspond to idempotents. For $s\notin\mathfrak m$, choose an idempotent $e$ with $\mathfrak m\in D(e)\subseteq D(s)$. Then $es$ is a unit in $eR$. This represents $s^{-1}$ at $\mathfrak m$ by an element of $R$, proving surjectivity. If $r\in I$, choose $s\notin\mathfrak m$ with $sr=0$ and such an $e$. Then $er=0$, so $r=(1-e)r$, and $1-e\in I$. Thus $I$ is generated by idempotents.
	
	Suppose $rI=0$. At each $\mathfrak n\ne\mathfrak m$, an idempotent in $I$ becomes one, so $r_{\mathfrak n}=0$. Coherence implies that $\ann_R(r)$ is finitely generated. Hence
	\[
	\operatorname{Supp}(Rr)=V(\ann_R(r))
	\]
	is clopen: for each element $b$, $D(b)$ is clopen in a zero-dimensional spectrum, and the displayed closed set is a finite intersection of their complements. If $r\ne0$, its support is therefore the open singleton $\{\mathfrak m\}$. This contradicts nonisolation. Consequently $r=0$.
\end{proof}

\begin{lemma}\label{lem:simple-local-envelope}
	Let $(A,\mathfrak n,k)$ be an Artinian local ring. Choose a $k$-basis $s_1,\ldots,s_t$ of $\operatorname{Soc}(A)=\ann_A(\mathfrak n)$. Then
	\[
	u:k\longrightarrow P=A^t,\qquad u(1)=(s_1,\ldots,s_t),
	\]
	is a projective envelope.
\end{lemma}
\begin{proof}
	The socle is nonzero and finite-dimensional. A map $k\to A^{(J)}$ sends $1$ to a vector with finite support whose coordinates belong to the socle. Expressing those coordinates in the indicated basis factors the map through $u$. All projective $A$-modules are free, so this proves the preenvelope property. If a matrix $H=(h_{ij})$ satisfies $Hu=u$, then independence of the $s_j$ gives $h_{ij}\equiv\delta_{ij}\pmod{\mathfrak n}$. Its determinant is a unit, so $H$ is invertible. This proves minimality.
\end{proof}

\begin{theorem}[Complete simple-module criterion]\label{thm:simple-criterion}
	Let $R$ be coherent and locally Artinian, let $\mathfrak m\in\Max R$, and put $k=R/\mathfrak m$ and $A=R_{\mathfrak m}$. The simple module $k$ has an almost projective envelope if and only if at least one of the following holds:
	\begin{enumerate}[label=(\roman*)]
		\item $A$ is a field;
		\item $\mathfrak m$ is an isolated point of $\Spec R$.
	\end{enumerate}
	Consequently, every simple $R$-module has an almost projective envelope if and only if every nonfield point is isolated.
\end{theorem}
\begin{proof}
	The ring $R$ is zero-dimensional. If $A$ is a field, then $k=A$, which is almost projective as an $R$-module: its stalk at $\mathfrak m$ is free and its other stalks vanish. Its identity map is an envelope. If $\mathfrak m$ is isolated, then $A=eR$ for an idempotent $e$. The map in Lemma~\ref{lem:simple-local-envelope}, extended by zero on $(1-e)R$, is an almost projective envelope of $k$. Factorization and minimality can both be checked on this direct factor.
	
	For necessity, assume that $\mathfrak m$ is nonisolated and $A$ is not a field. Write $I=\ker\pi$ for the surjection $\pi:R\to A$. Let $u:k\to P=A^t$ be as in Lemma~\ref{lem:simple-local-envelope}. If an almost projective envelope $f:k\to F$ existed, every idempotent in $I$ would annihilate $F$: multiplication by its complement fixes $f$, and an invertible idempotent is the identity. Since $I$ is generated by idempotents, $IF=0$. Thus $F$ is an $A$-module and is free over $A$. Every projective $A$-module is almost projective by restriction to $R$, because its other stalks vanish. Therefore $f$ is a projective envelope in $A$-modules. Uniqueness of envelopes identifies $f$ with $u$; in particular, $u$ would be an almost projective envelope over $R$.
	
	Choose $0\ne a\in\mathfrak n$, the maximal ideal of $A$. Since $a$ annihilates the socle, $au=0$. Define the pullback module
	\begin{equation}\label{eq:general-pullback}
		G=\{(x,v)\in R^t\oplus P:\ \pi^t(x)=av\}.
	\end{equation}
	It is the kernel of the surjection
	\[
	R^t\oplus P\longrightarrow P,\qquad (x,v)\longmapsto\pi^t(x)-av.
	\]
	At $\mathfrak m$ this kernel is the graph of multiplication by $a$ on $P$, and is isomorphic to $A^t$. At any other maximal ideal $\mathfrak q$, both copies of $P$ localize to zero, and $G_{\mathfrak q}\cong R_{\mathfrak q}^t$. Thus $G$ is almost projective. The map
	\[
	j:k\longrightarrow G,\qquad j(c)=(0,u(c))
	\]
	is well-defined.
	
	If $j$ factored through $u$, write the factor map $h:P\to G$ as $(x,v)$, where $x:P\to R^t$ and $v:P\to P$. Since $IP=0$, the image of $x$ is annihilated by $I$. Lemma~\ref{lem:supported-element} gives $x=0$. The defining equation of $G$ now gives $av=0$, while $hu=j$ gives $vu=u$. Minimality of the local projective envelope makes $v$ invertible. Hence $a=0$, a contradiction. Thus $u$ cannot even be an almost projective preenvelope over $R$, and $f$ cannot exist.
\end{proof}

\begin{corollary}[A zero-dimensional necessary condition]\label{cor:zero-necessary}
	Let $R$ be a zero-dimensional commutative ring. If every $R$-module has an almost projective envelope, then $R$ is coherent and locally Artinian, and every point $\mathfrak m$ for which $R_{\mathfrak m}$ is not a field is isolated.
\end{corollary}
\begin{proof}
	Coherence follows from Proposition~\ref{prop:coherence}. For every maximal ideal $\mathfrak m$ and every $R_{\mathfrak m}$-module $M$, Proposition~\ref{prop:coherence} turns an almost projective envelope over $R$ into a projective envelope over $R_{\mathfrak m}$. The classical projective-envelope theorem therefore makes $R_{\mathfrak m}$ Artinian. Theorem~\ref{thm:simple-criterion} now applies to all simple modules.
\end{proof}

\begin{corollary}[Finitely many non-QF stalks]\label{cor:finitely-nonqf}
	Let $R$ be coherent and locally Artinian. Suppose that only finitely many maximal localizations of $R$ fail to be QF. Then every $R$-module has an almost projective envelope if and only if every nonfield point is isolated.
\end{corollary}
\begin{proof}
	Necessity follows from Theorem~\ref{thm:simple-criterion}. Conversely, all non-QF points are nonfield points, and hence are isolated by hypothesis. Their finite union is clopen. The corresponding idempotent decomposes $R$ as a product $R_1\times R_2$, where $R_1$ is a finite product of Artinian local rings and $R_2$ is coherent and locally QF. The ring $R_1$ has projective envelopes for all modules, while $R_2$ has almost projective envelopes for all modules by Theorem~\ref{thm:qf}. The finite-product assertion from Section~\ref{sec:further} proves the result.
\end{proof}

This corollary recovers the local Artinian necessity under a stated zero-dimensional hypothesis and adds an isolation condition. It does not show that the zero-dimensional hypothesis follows from the existence of envelopes. Moreover, Theorem~\ref{thm:simple-criterion} classifies envelopes of simple modules, not of arbitrary modules. Its necessary isolation condition is sufficient for all modules in the locally QF case by Theorem~\ref{thm:qf}; a proof of that sufficiency without the QF assumption is not supplied here.

\section{Dual numbers over von Neumann regular rings and a counterexample}\label{sec:dual}

\subsection{Dual numbers over von Neumann regular rings}

The dual-number classification agrees with Theorem \ref{thm:qf}: all stalks are nonfield QF rings, so every point would have to be isolated. Compactness would then force finiteness. In this section, we will supply a direct module-theoretic proof, independent of the monomorphic-envelope classification.

Fix a commutative von Neumann regular ring $B$ and set
\[
 R=B[\eps]/(\eps^2).
\]
Recall that von Neumann regular means that for every $b\in B$ there is $c\in B$ with $b=b^2c$.   Its local rings are fields, every finitely generated ideal is generated by an idempotent, and every module is flat.

\begin{lemma}\label{lem:dual-ring}
The ring $R$ is coherent. Its maximal ideals are $\mathfrak M_{\mathfrak p}=\mathfrak pR+(\eps)$, for $\mathfrak p\in\Spec B$, and
\[
 R_{\mathfrak M_{\mathfrak p}}
 \cong k(\mathfrak p)[\eps]/(\eps^2),\qquad
 k(\mathfrak p)=B/\mathfrak p=B_{\mathfrak p}.
\]
In particular $R$ is locally Artinian, with local rings of composition length two.
\end{lemma}
\begin{proof}
The ring $B$ is coherent. For a map $R^a\to R^b$, regard it as a map between finite free $B$-modules of ranks $2a$ and $2b$. Its kernel is finitely generated over $B$, hence over $R$. The finite-free-kernel criterion gives coherence of $R$.

Every prime ideal contains the nilpotent $\eps$, so prime ideals correspond to those of $B$. Localizing the two-term $B$-basis $1,\eps$ gives the displayed isomorphism. The maximal ideal of the local ring on the right is generated by $\eps$, its square is zero, and its residue field and radical are both one-dimensional over $k(\mathfrak p)$.
\end{proof}

\begin{lemma}\label{lem:dual-free}
Let $K$ be a field, $A=K[\eps]/(\eps^2)$, and let $V$ be any $A$-module. Then $V$ is free over $A$ if and only if
\[
 \ker(\eps:V\to V)=\eps V.
\]
More generally every $A$-module is isomorphic to $A^{(I)}\oplus K^{(J)}$ for suitable sets $I,J$.
\end{lemma}
\begin{proof}
View multiplication by $\eps$ as a square-zero linear operator $d$ on a $K$-vector space. Choose a basis $(w_i)$ of $\im d$, lift it to vectors $v_i$ with $dv_i=w_i$, and extend $(w_i)$ to a basis $(w_i,z_j)$ of $\ker d$. Then $(v_i,w_i,z_j)$ is a basis of $V$: applying $d$ proves independence of the $v_i$ components, and subtracting appropriate $v_i$ components places every vector in $\ker d$. Each pair $(v_i,w_i)$ is a copy of $A$ and each $z_j$ is a copy of $K$. The equality $\ker d=\im d$ is exactly the absence of the latter summands.
\end{proof}

\subsection{The obstruction module}
Fix $\mathfrak p\in\Spec B$, let $k=B/\mathfrak p$, and write $\pi:B\to k$ for the quotient. Give
\begin{equation}\label{eq:G}
 G=B\oplus k\oplus\mathfrak p
=\end{equation}
the componentwise $B$-action and define
\begin{equation}\label{eq:d}
 \eps(b,c,d)=(d,\pi(b),0).
\end{equation}
This is $B$-linear and its square is zero because $\pi(d)=0$.

\begin{lemma}\label{lem:G-flat}
The module $G$ is almost projective and flat over $R$. The map
\[
 j:k\longrightarrow G,\qquad j(c)=(0,c,0),
\]
is an injective $R$-homomorphism, where $\eps k=0$.
\end{lemma}
\begin{proof}
Directly from \eqref{eq:d},
\[
 \ker\eps=\mathfrak p\oplus k\oplus0=\im\eps.
\]
Localization is exact, so the equality remains true at every maximal ideal of $R$. Lemmas~\ref{lem:dual-ring} and \ref{lem:dual-free} show that every localization of $G$ is free. Thus $G\in\AP(R)=\Flat(R)$. The assertions about $j$ follow immediately from its formula.
\end{proof}

\begin{lemma}\label{lem:nonfactor}
Put $A=R/\mathfrak pR=k[\eps]/(\eps^2)$ and $u:k\to A$, $u(c)=\eps c$. If $j:k\to G$ factors through $u$, then $k$ is projective as a $B$-module.
\end{lemma}
\begin{proof}
A factorization is an $R$-linear $g:A\to G$ with $g(\eps)=j(1)$. Write $g(1)=(b,c,d)$. Since $\mathfrak pA=0$, we have $\mathfrak p b=0$. Moreover
\[
 (d,\pi(b),0)=\eps g(1)=j(1)=(0,1,0).
\]
Hence $d=0$ and $\pi(b)=1$. The formula $\sigma(\pi(a))=ab$ defines a well-defined $B$-linear section of $\pi:B\to k$. Therefore $k$ is a direct summand of $B$ and is projective.
\end{proof}
This calculation is the global obstruction. Local freeness of $G$ supplies local preimages under $\eps$. It does not supply a preimage annihilated by the entire ideal $\mathfrak p$.

\begin{lemma}\label{lem:automorphism}
	Let $f:M\to C$ be a $\mathcal C$-envelope. Suppose $\alpha\in\operatorname{Aut}_R(M)$ and $\beta\in\operatorname{End}_R(C)$ satisfy $\beta f=f\alpha$. Then $\beta$ is an automorphism. 
\end{lemma}
\begin{proof}
	Factor $f\alpha^{-1}$ through $f$, obtaining $\gamma:C\to C$ with $\gamma f=f\alpha^{-1}$. Then
	\[
	(\gamma\beta)f=f=(\beta\gamma)f.
	\]
	Both $\gamma\beta$ and $\beta\gamma$ are automorphisms by minimality. Thus $\beta$ has both a left inverse $(\gamma\beta)^{-1}\gamma$ and a right inverse $\gamma(\beta\gamma)^{-1}$. These inverses agree.
\end{proof}

\begin{lemma}[Idempotent annihilation]\label{lem:idempotent}
	If $f:M\to C$ is an envelope and $e^2=e\in R$ satisfies $eM=0$, then $eC=0$.
\end{lemma}
\begin{proof}
	Multiplication by $1-e$ fixes $f$, hence is an automorphism of $C$. It is also idempotent. An invertible idempotent endomorphism is the identity, so multiplication by $e$ is zero.
\end{proof}

\begin{theorem}[A single residue field]\label{thm:residue}
With $B,R,\mathfrak p,k$ as above, the $R$-module $k$, annihilated by $\eps$, has an almost projective envelope if and only if $k$ is projective over $B$.
\end{theorem}
\begin{proof}
Suppose $f:k\to F$ is such an envelope. Every element of $\mathfrak p$ lies in an idempotent-generated ideal contained in $\mathfrak p$. Lemma~\ref{lem:idempotent} therefore gives $\mathfrak pF=0$. Thus $F$ is an $A=k[\eps]/(\eps^2)$-module. This quotient is the localization $R_{\mathfrak M_{\mathfrak p}}$, so $F=F_{\mathfrak M_{\mathfrak p}}$ and almost projectivity makes $F$ free over $A$.

The module $A=R\otimes_B k$ is flat over $R$, since $k$ is flat over $B$. Hence it is almost projective. The injection $u:k\to A$, $c\mapsto\eps c$, factors through $f$, proving that $f$ is injective. Write $F=A^{(I)}$. Since $\eps f(1)=0$, Lemma~\ref{lem:dual-free} gives $f(1)\in\eps F$. It is a nonzero vector with finite support. A basis change over $A$ therefore makes
\[
 f(1)=\eps v
\]
for a basis vector $v$ of $F$. For completeness, write the finitely many nonzero coefficients as $\eps c_i$ with $c_i\in k$ and some $c_i\ne0$; elementary linear changes over $k\subseteq A$ transform this vector to the first coordinate.

Projection onto $Av$ fixes $f$. Envelope minimality forces this projection to be invertible. Hence $F=Av\cong A$, and $f$ is isomorphic to $u$. The map $j:k\to G$ must factor through $u$, since $G$ is almost projective. Lemma~\ref{lem:nonfactor} now makes $k$ projective over $B$.

Conversely, if $k$ is projective over $B$, the quotient $B\to k$ splits and there is an idempotent $e\in B$ with $eB\cong k$ and $\mathfrak p=(1-e)B$. Thus $eR\cong A$ is a direct product factor of $R$. For an almost projective target $P$, its component $eP$ is free over the local factor $A$. Any map $k\to P$ lands in $eP$ and then in $\ann_{eP}(\eps)=\eps eP$, so it extends along $u:k\to A$. An endomorphism of $A$ fixing $u$ is multiplication by $1+\eps a$, a unit. Thus $u$ is an almost projective envelope over $R$.
\end{proof}

\begin{lemma}\label{lem:regular-finite}
For a commutative von Neumann regular ring $B$, every residue field is projective over $B$ if and only if $B$ is a finite product of fields.
\end{lemma}
\begin{proof}
If $B/\mathfrak p$ is projective, the cyclic quotient splits, so $B/\mathfrak p=eB$ for an idempotent $e$. Its support is the single point $\mathfrak p$, making that point open in $\Spec B$. If every residue field is projective, the spectrum is discrete. Since spectra are quasi-compact, it is finite. A reduced zero-dimensional ring with finite spectrum is a finite product of fields, by the Chinese remainder theorem. The reverse implication is immediate from the finite product decomposition.
\end{proof}

\begin{lemma}\label{lem:dual-envelopes}
Every module over $A=K[\eps]/(\eps^2)$ has a projective envelope.
\end{lemma}
\begin{proof}
Write a module as $M=A^{(I)}\oplus K^{(J)}$ by Lemma~\ref{lem:dual-free}. Define
\[
 f:M\longrightarrow P=A^{(I)}\oplus A^{(J)},\qquad
 f(a,c)=(a,\eps c).
\]
For a free target, the images of the $K$-basis vectors lie in its $\eps$-multiple submodule, so choose preimages and extend the map from $M$ to $P$. Since every projective over $A$ is free, $f$ is a projective preenvelope.

If $hf=f$, then $h$ fixes the $I$-summand and induces the identity modulo $\eps$ on the $J$-basis vectors. Hence $h=1+t$ with $t(P)\subseteq\eps P$. By $A$-linearity, $t(\eps P)=\eps t(P)=0$, so $t^2=0$ and $(1+t)^{-1}=1-t$. Thus $f$ is an envelope, for arbitrary sets $I,J$.
\end{proof}

\begin{proof}[Proof of Theorem~\ref{thm:main-dual}]
Lemma~\ref{lem:dual-ring} proves coherence and local Artinianness. The implication (i)$\Rightarrow$(ii) is immediate. Theorem~\ref{thm:residue} proves (ii)$\Leftrightarrow$(iii), and Lemma~\ref{lem:regular-finite} proves (iii)$\Leftrightarrow$(iv). If $B$ is a finite product of fields, $R$ is a finite product of rings of dual numbers. Apply Lemma~\ref{lem:dual-envelopes} componentwise. Since almost projective modules over this semilocal ring are exactly projectives, the resulting envelopes are almost projective envelopes. Thus (iv)$\Rightarrow$(i).
\end{proof}

\subsection{A  counterexample}

In Proposition  \ref{prop:positive}, we have shown that suppose $R$ is coherent locally Artinian, then every module has an almost projective envelope under the following  conditions:
\begin{enumerate}[label=(\roman*)]
	\item $R$ is reduced;
	\item $R$ has finite weak global dimension;
	\item $R$ is Noetherian;
	\item $R$ is semilocal.
\end{enumerate}

A natural question is that can we omit these conditions? The following example give a negative answer to this question.

\begin{example}\label{ex:eventual}
Let $k$ be a field and let
\[
 B=\{(b_n)_{n\geq1}\in k^{\N}: (b_n)\text{ is eventually constant}\}.
\]
This is a countable von Neumann regular ring. Let
\[
 \mathfrak p=k^{(\N)}\subseteq B,\qquad
 \pi:B\to k
\]
be the eventual-value map. The ideal $\mathfrak p$ consists of the finitely supported sequences. The quotient $k$ is not projective: if the quotient map split, there would be $b\in B$ with $\pi(b)=1$ and $\mathfrak p b=0$. Multiplication by each coordinate idempotent $e_n\in\mathfrak p$ would give $b_n=0$ for all $n$, a contradiction.

Consequently
\begin{equation}\label{eq:counter-ring}
 R=B[\eps]/(\eps^2)
\end{equation}
is a countable coherent locally Artinian ring and
\[
 k=R/(\mathfrak pR+\eps R)
\]
has no almost projective envelope. The obstructing map is explicitly
\[
 k\longrightarrow B\oplus k\oplus\mathfrak p,\qquad
 c\longmapsto(0,c,0),
\]
with the action \eqref{eq:d}.
\end{example}

The maximal ideals of $B$ are the coordinate kernels and the eventual-value kernel $\mathfrak p$. Indeed, a prime quotient in which some $e_n$ survives is the $n$th coordinate field; if all $e_n$ vanish, the quotient factors through $B/\mathfrak p=k$. At a coordinate point, $G$ has underlying localized vector space $k\oplus0\oplus k$ with $\eps(b,0,d)=(d,0,0)$. At the eventual-value point it is $k\oplus k\oplus0$ with $\eps(b,c,0)=(0,b,0)$. Thus the obstruction target is free of rank one at \emph{every} maximal ideal. All local rings of $R$ are the same Artinian ring $k[\eps]/(\eps^2)$; the failure is genuinely global.

\begin{corollary}\label{cor:negative}
Coherence and local Artinianness do not imply that every module has an almost projective envelope. They do imply that every module has an almost projective preenvelope. The failure already occurs for a simple cyclic module over a countable ring whose local rings all have length two.
\end{corollary}
\begin{proof}
Combine Theorem~\ref{thm:preenvelopes} with Example~\ref{ex:eventual}.
\end{proof}

\end{document}